\documentclass[11pt]{amsart}
\usepackage{amssymb,amsmath,amsthm,amscd,mathrsfs,graphicx}
\usepackage[cmtip,all]{xy}
\usepackage{pb-diagram}
\usepackage{tikz}
\usepackage{pgfplots}

\numberwithin{equation}{section}

\newtheorem{theorem}{Theorem}[section]

\newtheorem{lemma}{Lemma}[section]
\newtheorem{corollary}{Corollary}[section]

\newcommand{\ov}[1]{\overline{#1}}

\newcommand{\ve}{\varepsilon}

\theoremstyle{definition}

\usepackage{lipsum}

\theoremstyle{remark}

\begin{document}
\bibliographystyle{amsplain}

\author[A. Chau]{Albert Chau}
\address{Department of Mathematics, The University of British Columbia, 1984 Mathematics Road, Vancouver, B.C.,  Canada V6T 1Z2.  Email: chau@math.ubc.ca. } 
\author[B. Weinkove]{Ben Weinkove}
\address{Department of Mathematics, Northwestern University, 2033 Sheridan Road, Evanston, IL 60208, USA.  Email: weinkove@math.northwestern.edu.}

\title[{\large $\alpha$}-concavity for the porous medium equation]{Non-preservation of {\large $\alpha$}-concavity for \\ the porous medium equation}

\thanks{Research supported in part by  NSERC grant $\#$327637-06 and NSF grant DMS-2005311. Part of this work was carried out while the second-named author was visiting the Department of Mathematical Sciences at the University of Memphis and he thanks them for their kind hospitality.}

\maketitle

\vspace{-20pt}

\begin{abstract}
We show that the porous medium equation does not in general preserve $\alpha$-concavity of the pressure for $0 \le \alpha <1/2$ or $1/2<\alpha \le 1$.  In particular, this resolves an open problem of V\'azquez on whether concavity of pressure is preserved by the porous medium equation.  Our results 
strengthen an earlier work of Ishige-Salani, who considered the case of small $\alpha>0$.  Since Daskalopoulos-Hamilton-Lee showed that $1/2$-concavity is preserved, our result is sharp.

Our explicit examples show that concavity can be instantaneously broken at an interior point of the support of the initial data.  For $0\le \alpha<1/2$, we give another set of examples to show that concavity can be broken at a boundary point.
\end{abstract}

\section{Introduction}
 
The porous medium equation (PME) is a model for gas diffusing in a porous medium.  The gas density  $u(x,t)\ge 0$ satisfies the equation
\begin{equation}\label{PME}
\begin{split}
\frac{\partial u}{\partial t} = {} &  \Delta (u^m) \\
\end{split}
\end{equation}
for $(x,t)\in \mathbb{R}^n \times (0,\infty)$ where $m>1$ is a given number, with initial data
$$u(x,0)=u_0(x) \ge 0, \quad \textrm{for } x \in \mathbb{R}^n.$$
The equation (\ref{PME}) is degenerate where $u=0$, and so has to be interpreted in a weak sense (described below in Section \ref{subsec}) which requires only that $u_0$ is nonnegative and in $L^1(\mathbb{R}^n)$.  In the case where $u_0$ is H\"older continuous and compactly supported, the PME has a unique continuous solution $u(x,t)$ on $\mathbb{R}^n \times [0,\infty)$ in this weak sense.  In general, for each $t>0$, the function $x \mapsto u(x,t)$ is smooth on the set where ``there is gas'', given by
 $$\Omega_t = \{ x \in \mathbb{R}^n \ | \ u(x,t)>0\}.$$
 The boundary $\partial \Omega_t$ is the ``free boundary'' of this equation.
 
 It is convenient to work with the \emph{pressure} function $v:=(m/(m-1))u^{m-1}$ which satisfies 
 the evolution equation
  \begin{equation}\label{PME2}
\begin{split}
\frac{\partial v}{\partial t} = {} & (m-1)v\Delta v+|\nabla v|^2
\end{split}
\end{equation}
on the set $\Omega_t$ for all $t$ and with initial data $v(x,0)= v_0(x)$ supported on $\Omega=\Omega_0$.  In what follows, we will also refer to the equation for pressure (\ref{PME2}) as the porous medium equation.

It has long been of interest to understand whether concavity/convexity conditions are preserved by parabolic evolution equations, both with and without a free boundary.  We define the \emph{$\alpha$-concavity} of a function $f$  for $\alpha> 0$ to mean that $f^{\alpha}$ is concave, while $0$-concavity asserts that $\log f$ is concave.  In particular, if $f$ is twice differentiable then $\alpha$-concavity for $\alpha\ge 0$  means that the matrix with $(i,j)$th entry
$$f D_i D_j f -(1-\alpha)D_if D_j f$$
is nonpositive. 

The log concavity of positive solutions of the heat equation on a fixed domain is preserved \cite{BL}, whereas convexity of the level sets on ring domains is not \cite{CW1, IS1}, except under special conditions \cite{B, CW2, CH, CMS, IS3}.  For solutions of the  one-phase Stefan problem, which satisfy  the heat equation in the interior and for which the (free) boundary moves in the outward normal direction with speed equal to the norm of the gradient, $\alpha$-concavity is not preserved for $0\le \alpha <1/2$ \cite{CW3}.

It is natural to ask for which values of $\alpha$ is $\alpha$-concavity preserved along the porous medium equation.
Dasakopoulos-Hamilton-Lee \cite{DHL} showed that root concavity of the pressure $v$, corresponding to $\alpha=1/2$, is preserved for the porous medium equation.  More precisely, they considered initial data $v_0$ which is root concave, smooth up to the boundary of the convex set $\Omega$, and satisfies the nondegeneracy condition $v_0+|Dv_0| \ge c>0$ on $\Omega$.   They showed that  the solution $v(x,t)$ is root concave for all $t>0$.  

On the other hand, Ishige-Salani \cite{IS2} showed that there exists some $\alpha$ with $0<\alpha<1/2$ such that $\alpha$-concavity is \emph{not} in general preserved.  However the case of general $\alpha \in [0,1] \setminus \{ \frac{1}{2} \}$, including $\alpha=1$, has remained open until now.
  V\'azquez, in his 2007 monograph \cite[p. 520]{V}, posed the open problem: ``Prove or disprove the preservation of pressure concavity for the solutions of the PME in several space dimensions.'' Note that for $n=1$, B\'enilan and V\'azquez \cite{BV} had already shown that concavity of pressure \emph{is} preserved.



Our first result shows that $\alpha$-concavity is \emph{not preserved} in general by the porous medium equation in dimension $n=2$ for any $\alpha$ in $[0,1] \setminus \{ \frac{1}{2}\}$.  By the result of \cite{DHL}, this  is sharp.

  \begin{theorem} \label{maintheorem0}
Let $B$ be the open unit ball in $\mathbb{R}^2$ centered at the origin.  Given  $\alpha \in [0,1]\setminus \{\frac{1}{2}\}$, there exists $v_0 \in C^{\infty}(\ov{B})$ which is strictly positive on $B$ and vanishes on $\partial B$ with the following properties:
\begin{enumerate}
\item[(i)] $v_0$ is $\alpha$-concave on $B$.
\item[(ii)] $\nabla v_0$ does not vanish at any point of $\partial B$.
\item[(iii)] Let $v(t)$ be the solution of the porous medium equation (\ref{PME2}) starting at $v_0$.   Then there exists $\delta>0$ such that $v(t)$ is not $\alpha$-concave in a neighborhood of the origin for $t \in (0, \delta)$.
\end{enumerate}
\end{theorem}
 
In particular this resolves the open problem of V\'azquez. 

Theorem \ref{maintheorem0} shows that $\alpha$-concavity can be instantaneously lost at an \emph{interior} point of the domain, for $\alpha \in [0,1] \setminus \{\frac{1}{2} \}$.  It is also natural to ask whether concavity breaking can occur at a \emph{boundary} point.  Namely, can the sets $\Omega_t$ be non-convex for $t>0$ if the initial data is $\alpha$-concave on a convex $\Omega$?   The work of Ishige-Salani \cite{IS2} showed that this phenomenon can occur for some $\alpha$ with $0<\alpha<1/2$.   Our second result extends the Ishige-Salani result to all $\alpha$ with $0\le \alpha<1/2$.
 
 \pagebreak[3]
 \begin{theorem} \label{maintheorem}
Given  $\alpha \in [0,1/2)$, there exist a bounded open convex set $\Omega \subset \mathbb{R}^2$ with smooth boundary and $v_0 \in C^{\infty}(\ov{\Omega})$ which is strictly positive on $\Omega$ and vanishes on $\partial \Omega$ with the following properties:
\begin{enumerate}
\item[(i)] $v_0$ is $\alpha$-concave on $\Omega$.
\item[(ii)] $\nabla v_0$ does not vanish at any point of $\partial \Omega$.
\item[(iii)] Let $v(x,t)$ be the solution of the porous medium equation (\ref{PME2}) starting at $v_0$, and write $\Omega_t = \{ x \in \mathbb{R}^n \ | \ v(x,t)>0\}$.  Then there exists $\delta>0$ such that 
$$\Omega_t \textrm{ is not convex for any $t \in (0,\delta)$}$$
and thus
$$v|_{\Omega_t} \textrm{ is not $\alpha$-concave for any $t\in (0,\delta)$}.$$
\end{enumerate}
\end{theorem}

The result of Daskalopoulos-Hamilton-Lee \cite{DHL} implies in particular that if the initial data is $\alpha$-concave with $\alpha \ge 1/2$, the domain $\Omega_t$ remains convex for all $t$.  Hence Theorem \ref{maintheorem} is sharp in this sense.

The proof of Theorem \ref{maintheorem} uses as initial data $v_0$ the function constructed in \cite{CW3} which was originally applied to the case of the one-phase Stefan problem.  The free boundary for the solution $v(x,t)$ of the one-phase Stefan problem moves in the outward normal direction with speed $|\nabla v|$, given certain compatibility conditions for the initial data.  For the porous medium equation, the free boundary also moves with speed equal to $|\nabla v|$ for solutions which are smooth up to the free boundary, and so the ``formal'' proof of Theorem \ref{maintheorem} is no different from the argument in \cite{CW3}.  

However, the precise result we require, giving the speed of the boundary at the initial time $t=0$, does not seem to be explicitly stated in the literature (see for example \cite{CF1, CF2, CVW, DH, K} for some closely related results).  For this reason, we provide a simple proof of the result that we need (Theorem \ref{thmballs} below) which gives bounds on the velocity of a point on the boundary from the initial time, assuming interior and exterior ball conditions for the initial data.  This result, which uses only the well-known Barenblatt and Graveleau solutions to the PME and a local comparison theorem, may be of independent interest.

The outline of our paper is as follows.  In Section \ref{prelim} we  state some well-known results about the porous medium equation and recall the construction of \cite{CW3}.  In Section \ref{interior} we prove Theorem \ref{maintheorem0}, the main result of this paper.  In Section \ref{BG} we recall the Barenblatt and Graveleau solutions of the porous medium equation and their basic properties.  These are used to prove a general result about the velocity of the boundary, proved in Section \ref{velocity}, which is then used to establish Theorem \ref{maintheorem} in Section \ref{proofmainthm}.
 
\pagebreak[3]
\section{Preliminaries} \label{prelim}

\subsection{Solutions of the porous medium equation} \label{subsec}  Let $u_0$ be a nonnegative function in $L^1(\mathbb{R}^n)$, and let $m>1$.
 Fix $0<T\le \infty$.  A solution to the porous medium equation (\ref{PME}) for time $t \in [0,T)$ with initial data $u_0$ is a nonnegative function $u(x,t)$ in $C([0,T) : L^1(\mathbb{R}^n))$ such that 
$$u^m \in L^1_{\textrm{loc}}([0,T) : L^1(\mathbb{R}^n)),  \quad \textrm{and}\quad u_t, \Delta u^m \in L^1_{\textrm{loc}}(\mathbb{R}^n \times (0,T)),$$
with
$$u_t = \Delta u^m, \quad \textrm{a.e. in } \mathbb{R}^n \times (0,T)$$
and 
$$u(t) \rightarrow u_0 \ \textrm{in } L^1(\mathbb{R}^n) \quad \textrm{as } t\rightarrow 0.$$
There exists a unique such solution $u(x,t)$ with any initial data $u_0$ as above  \cite{BC, P}, known as  a \emph{strong $L^1$ solution} (see the expositions \cite[Theorems 9.2, 9.12]{V} and \cite[Section 5.1]{F} and the references therein).  
If, in addition, we assume that $u_0$ is compactly supported and H\"older continuous with H\"older exponent $\gamma>0$ say, then 
the solution $u(x,t)$ is continuous on $\mathbb{R}^n \times [0,\infty)$ \cite{CF1}.  On the set $\{ u>0 \}\cap \{ t>0 \}$, the solution $u$ is smooth.  Moreover,  this smoothness extends to $t=0$ if the initial data $u_0$ is smooth on the set $\{ u_0 >0 \}$ (see  \cite[Proposition 7.21]{V} and the remark afterwards).

Solutions of the porous medium equation satisfy a local comparison principle  \cite{ACP, DK}, which we state without proof.  We give this in terms of locally defined continuous solutions, since this is the context in which we will need to use it later. Let $B$ be a ball in $\mathbb{R}^n$ and let 
$u(x,t)$, $\tilde{u}(x,t)$ be nonnegative continuous functions on $\ov{B} \times [0,T)$ satisfying $u_t, \Delta u^m \in L^1_{\textrm{loc}}(B \times (0,T))$ and
$$u_t = \Delta u^m, \quad \textrm{a.e. in } B \times (0,T),$$
with initial data $u(x,0)= u_0(x)$ and $\tilde{u}(x,0)=\tilde{u}_0(x)$ for $x \in \ov{B}$.
Then we have the following comparison theorem (see for example \cite[Lemma 9.30]{V}).

\begin{theorem} \label{comparison}
Let  $u, \tilde{u}$ be as above.  Assume that 
$$u_0(x) \le \tilde{u}_0(x), \quad \textrm{for } x\in \ov{B}$$
 and 
 $$u(x,t) \le \tilde{u}(x,t), \quad  \textrm{for } (x,t) \in \partial B \times [0,T).$$
  Then 
$$u(x,t) \le \tilde{u}(x,t), \quad \textrm{for } (x,t) \in B \times [0,T).$$
\end{theorem}

\subsection{Initial data for boundary concavity breaking}  The authors constructed a specific $\alpha$-concave function supported on a bounded convex domain \cite[Theorem 3.1]{CW3}:

\begin{theorem}  \label{thm1} For each $\alpha \in [0,1/2)$, 
there is a non-negative function $v_0$ on $\mathbb{R}^2$ and an open bounded convex set $\Omega$ lying in the upper half plane such that its boundary $\partial \Omega$ contains the line segment $\{(x, y): x\in [-1, 1], y=0\}$, and 
\begin{enumerate}
\item[(a)] $v_0$ is positive on $\Omega$, vanishes on $\Omega^c$, and is smooth on $\ov{\Omega}$.
\item[(b)] $v_0$ is $\alpha$-concave on $\Omega$.
\item[(c)] $\nabla v_0$ does not vanish at any point of $\partial \Omega$.
\item[(d)] The map
$$x\mapsto \frac{\partial v_0}{\partial y}(x, 0), \qquad \textrm{for } x \in [-1, 1],$$
is positive and strongly convex.
\end{enumerate}
\end{theorem}

 This theorem will give us the initial data for our solution in Theorem \ref{maintheorem}.

\section{Interior concavity breaking} \label{interior}

 In this section we give the proof of Theorem \ref{maintheorem0}.

We begin by computing some evolution formulae for smooth positive solutions of the porous medium equation.  Let $v$ be a local smooth positive solution of (\ref{PME2}).  For $\alpha>0$, write $w= v^{\alpha}$ and compute
\begin{equation} \label{eqnw1}
\frac{\partial w}{\partial t} = (m-1) w^{1/\alpha} w_{kk} + \frac{1}{\alpha} (1 + (m-1)(1-\alpha)) w^{1/\alpha-1} w_k^2,
\end{equation}
where  subscripts  denote (spatial) partial derivatives, and by the usual convention we are omitting the summation sign in the  $k$ indices.

Our example relies crucially on the evolution of the second derivative of $w$ in a fixed direction, which without loss of generality we take to be $x_1$.  Differentiating (\ref{eqnw1}) twice with respect to $x_1$ and simplifying gives the following.
\begin{equation}\label{w11talpha}
\begin{split}
\frac{\partial}{\partial t}w_{11} = {} & (m-1) w^{1/\alpha} w_{kk11} + \frac{2(m-1)}{\alpha} w^{1/\alpha-1} w_{1} w_{kk1} \\ {} &  + \frac{(m-1)}{\alpha} \left( \frac{1}{\alpha} -1 \right) w^{1/\alpha-2} w_1^2 w_{kk}   +   \frac{(m-1)}{\alpha} w^{1/\alpha-1}w_{11} w_{kk}  \\
 {} & +\frac{1}{\alpha} \left( 1 + (m-1) (1-\alpha) \right) \bigg\{   \left( \frac{1}{\alpha} -2 \right) \left( \frac{1}{\alpha} -1 \right)w^{1/\alpha-3}w^2_1 w_k^2 \\ 
 {} & + \left( \frac{1}{\alpha} -1 \right)w^{1/\alpha-2}w_{11} w_k^2 +  4 \left( \frac{1}{\alpha} -1 \right)w^{1/\alpha-2}w_1 w_k w_{k1} \\
 {} & + 2w^{1/\alpha-1}w^2_{1k} + 2w^{1/\alpha-1}w_kw_{k11} \bigg\}.
 \end{split}\end{equation}

For the case $\alpha=0$ we now write $w =\log v$ and compute
\begin{equation} \label{eqnw2}
\frac{\partial w}{\partial t} = (m-1) e^w w_{kk} + me^w w_k^2,
\end{equation}
and
\begin{equation} \label{w11tlog}
\begin{split}
\frac{\partial}{\partial t}w_{11} = {} & (m-1) e^w w_{kk11} + 2(m-1)e^w w_1 w_{kk1} + (m-1) e^w w_1^2 w_{kk} \\
{} & + (m-1) e^w w_{11} w_{kk} +
me^w w_1^2 w_k^2 + m e^w w_{11} w_k^2 + 4m e^w w_1w_k w_{k1} \\ {}&  +2me^w w_{1k}^2 + 2me^w w_k w_{k11}.
\end{split}
\end{equation}

The main technical result of this section is the following lemma.  It shows that for $\alpha \in [0,1] \setminus \{ \frac{1}{2} \}$ we can find a positive concave function $w$ defined in a small ball such that, at the origin, we have $w_{11}=0$ and $\partial_t w_{11}>0$, in the sense of (\ref{w11talpha}) and (\ref{w11tlog}) above.  Write $B_{\rho}(0)$ for the open ball of radius $\rho$ in $\mathbb{R}^2$ centered at the origin.

\begin{lemma}\label{localinitialdata}
 For each $\alpha \in [0,1]\setminus \{ \frac{1}{2}\}$, there exists a smooth positive function $w$ on $\ov{B_{\rho}(0)}$ for some $\rho>0$ satisfying
 \begin{enumerate}
\item[(i)] $\displaystyle{w_{11}(0)=0, \ w_{22}(0)<0, \ w_{12}(0)=0}$.
\vspace{1pt}

\item[(ii)] $\displaystyle{(D^2w) < 0 \textrm{ on } \ov{B_{\rho}(0)}\setminus \{ 0\}}$.
\vspace{1pt}

\item[(iii)] If $\alpha\neq 0$ then the right hand side of \eqref{w11talpha} is positive at the origin.  If $\alpha=0$ then the right hand side of  \eqref{w11tlog} is positive at the origin.
\end{enumerate}

\end{lemma} 

\begin{proof} We consider two separate cases.

\bigskip
\noindent
{\bf Case 1.} \  $0\le \alpha<1/2$ or $\alpha =1$.
\bigskip

For a positive constant $a$ to be determined, define
$$w =1+  a x_1 - x_2^2 +  x_1 x_2^2 - x_1^4 -2x_1^2x_2^2,$$
on $\ov{B_{\rho}(0)}$, where $\rho>0$ is also to be determined.
  Then we have
$$(D^2w) = \begin{pmatrix} -12  x_1^2  -4x_2^2 & 2x_2 - 8x_1x_2 \\ \ & \ \\ 2x_2 - 8x_1x_2 & -2 +2x_1 - 4x_1^2 \end{pmatrix}.$$
Condition (i) is easily verified.   
For (ii), we note that 
$$\det(D^2 w) = 24x_1^2 + 4x_2^2 +  \textrm{O}( (x_1^2+x_2^2)^{3/2}),$$  and hence, choosing $\rho>0$ sufficiently small, we have $\det(D^2 w)>0$ and $w_{11}<0$ in $\ov{B_{\rho}(0)} \setminus \{0 \}$ which together imply (ii).

To verify condition (iii), observe that at the origin,
\begin{equation} \label{wcalc}
\begin{split}
& w_{1111}=-24, \  w_{2211}=-8, \  w_{221} =2, \  w_{22}=-2, \  w_1 =a, \ w=1,\\
& w_{111}=w_{211}=w_{11}=w_{12}=w_2=0.
\end{split}
\end{equation}
Assume $\alpha \neq 0$.  The right hand side of \eqref{w11talpha} is positive at the origin precisely when
\begin{equation}\label{ee1}
\begin{split} &  -32(m-1) + \frac{4(m-1)}{\alpha} a - \frac{2(m-1)}{\alpha} \left(\frac{1}{\alpha}-1 \right) a^2 \\
 {} & + \frac{1}{\alpha} \left( \frac{1}{\alpha} -2 \right) \left( \frac{1}{\alpha} -1 \right) \left( 1 + (m-1) (1-\alpha) \right) a^4 >0.
\end{split}
\end{equation}
If $0<\alpha<1/2$ we see that the last term on the left hand side of (\ref{ee1}) is strictly positive and so we choose $a$ sufficiently large so that this term dominates.  The inequality \eqref{ee1} and thus condition (iii) holds. If $\alpha =1$ then only the first two terms on the left hand side of (\ref{ee1}) are nonzero, and so it suffices to choose $a>8$.
 
 For $\alpha=0$, using again (\ref{wcalc}), the right hand side of \eqref{w11tlog} is positive at the origin when
 \begin{equation} \label{ee3}
 -32(m-1) + 4(m-1) a - 2a^2 (m-1) + ma^4>0,
 \end{equation}
 and this holds for $a$ sufficiently large.  This confirms (iii) in this case.
 
 To ensure that $w$ is positive on $\ov{B_{\rho}(0)}$ we note that $w(0)=1$ and shrink $\rho$ if necessary.

\bigskip
\noindent
{\bf Case 2.} \  $1/2<\alpha < 1$.  
\bigskip

 Consider for a constant $b>0$ to be determined,
$$w =1+\frac{ \alpha (3/2-\alpha)^{1/2}}{b(1-\alpha)} x_1 - b^2 x_2^2+b (3/2-\alpha)^{1/2} x_1x_2^2-\frac{x_1^4}{12b^2}-x_1^2x_2^2.$$  Then we have
$$(D^2w) = \begin{pmatrix} -x_1^2/b^2 -2x_2^2 \quad & 2b(3/2-\alpha)^{1/2}x_2-4x_1x_2 \\ \ & \ \\ 2b(3/2-\alpha)^{1/2}x_2-4x_1x_2 \quad & -2b^2 + 2b(3/2-\alpha)^{1/2}x_1-2x_1^2 \end{pmatrix},$$ 
from which (i) follows.  For (ii), compute
$$\det(D^2 w) = 2x_1^2 + 4b^2(\alpha-1/2)x_2^2  + \textrm{O}( (x_1^2+x_2^2)^{3/2}).$$ 
Choosing $\rho>0$ sufficiently small, we have $\det(D^2 w)>0$ and $w_{11}<0$ in $\ov{B_{\rho}(0)} \setminus \{0 \}$, giving (ii).

For (iii), we note that at the origin we have
\begin{equation} \label{wcalc2}
\begin{split}
& w_{1111}=-2/b^2, \  w_{2211}=-4, \  w_{221} =2b(3/2-\alpha)^{1/2}, \\ 
&  w_{22}=-2b^2, \  w_1 =\frac{\alpha(3/2-\alpha)^{1/2}}{b(1-\alpha)}, \ w=1,\\
& w_{111}=w_{211}=w_{11}=w_{12}=w_2=0.
\end{split}
\end{equation}
Then a short calculation shows that the right hand side of \eqref{w11talpha} is positive at the origin precisely  when 
\begin{equation}\label{ee2}\begin{split}
{} & -(m-1) \left( \frac{2}{b^2} + 4 \right) + \frac{2(m-1)(3/2-\alpha)}{1-\alpha}  \\
{} & + \frac{(1-2\alpha) (1+(m-1)(1-\alpha)) \alpha (3/2-\alpha)^2}{b^4(1-\alpha)^3} >0.
\end{split}\end{equation}
We can write the left hand side of (\ref{ee2}) as
$$\frac{2(m-1)}{1-\alpha} \left( \alpha - \frac{1}{2} - \frac{1-\alpha}{b^2}\right) - \frac{C_{m, \alpha}}{b^4},$$
for a positive constant $C_{m, \alpha}$ depending only on $m, \alpha$.  Since $1/2<\alpha<1$, we may choose $b$ sufficiently large, depending only on $\alpha$ and $m$, so that the inequality (\ref{ee2}) holds.
\end{proof}

We can now give the proof of Theorem \ref{maintheorem0}.

\begin{proof}[Proof of Theorem \ref{maintheorem0}]
Fix $\alpha \in [0,1]\setminus \{1/2\}$ and let $w$ be the function constructed in Lemma \ref{localinitialdata} which is positive on $\ov{B_{\rho}(0)}$ for some $\rho>0$.  We will use $w$ to construct a function $v_0$ with the properties listed in Theorem \ref{maintheorem0} on $B_{\rho}(0)$ instead of the unit ball $B$.  The theorem will then follow after scaling.

First consider the case $\alpha \neq 0,1$.  We define a smooth auxiliary function $F$ on $\ov{B_{\rho}(0)}$ as follows.  Let $f: [0,\rho] \rightarrow \mathbb{R}$ be a continuous decreasing concave function, smooth on $[0,\rho)$ satisfying
$$f(r) = \left\{ \begin{array}{ll} c_{\rho,\alpha} & \quad 0\le r \le \frac{\rho}{4} \\ (\rho-r)^{\alpha} & \quad \frac{\rho}{2} \le r\le \rho, \end{array} \right.$$
for $c_{\rho,\alpha}$ a constant slightly larger than $f(\rho/2)$, which we can take to be:
$$c_{\rho, \alpha} = \left( 1+\frac{\alpha}{4} \right) \left( \frac{\rho}{2} \right)^{\alpha}.$$
Then define $F: \ov{B_{\rho}(0)} \rightarrow \mathbb{R}$ by $F(x) = f(|x|)$, which is a concave function on $\ov{B_{\rho}(0)}$, smooth on $B_{\rho}(0)$ and vanishing on $\partial B_{\rho}(0)$.  Notice also that all the derivatives of $F$ vanish in the quarter-sized ball $B_{\rho/4}(0)$.  We have $f', f''<-\frac{1}{C}<0$ on the interval $[\rho/2, \rho)$ for a uniform positive constant $C$, by which we mean a constant that depends only on $\rho$ and $\alpha$.  Hence 
\begin{equation} \label{D2F}
D^2 F \le - \frac{1}{C} \textrm{Id}, \qquad \textrm{on } B_{\rho}(0) \setminus B_{\rho/2}(0),
\end{equation}
after increasing $C$ if necessary (see for example \cite[Lemma 2.1]{CW3}).

Now let $\psi: \ov{B_{\rho}(0)} \rightarrow [0,1]$ be a smooth cut-off function, equal to $1$ on $\ov{B_{\rho/2}(0)}$ and equal to $0$ on $\ov{B_{\rho}(0)}\setminus B_{3\rho/4}(0)$.  We may assume that 
\begin{equation} \label{Dpsi}
|D\psi| + |D^2\psi| \le C,
\end{equation}
 for a uniform constant $C$.

Define a function $\tilde{w}: \ov{B_{\rho}(0)} \rightarrow \mathbb{R}$ by
$$\tilde{w} = AF + \psi w,$$
for $A$ a large positive constant to be determined.  Compute
$$D^2 \tilde{w} =  AD^2 F + \psi D^2 w + (D^2 \psi ) w + 2D\psi \cdot D w.$$
On $B_{\rho/2}(0)$, $\psi$ is identically equal to 1 and hence  
$$D^2 \tilde{w} = AD^2 F + D^2 w \le 0,$$
as matrices, and $D^2\tilde{w} <0$ on $B_{\rho/2}(0)\setminus \{ 0 \}$.
  On $B_{\rho}(0) \setminus B_{\rho/2}(0)$ we have  from (\ref{D2F}) and (\ref{Dpsi}),
$$D^2 \tilde{w} \le \left( C  - \frac{A}{C} \right) \textrm{Id} < 0,$$
if we choose $A$ sufficiently large, depending only on $\rho$ and $\alpha$.  Hence $\tilde{w}$ is concave on $B_{\rho}(0)$ and has $D^2 \tilde{w} <0$ on the punctured ball $B_{\rho}(0)\setminus \{ 0 \}$.

We now define our initial data $v_0$ on $\ov{B_{\rho}(0)}$ by
$$v_0 = \tilde{w}^{\frac{1}{\alpha}}.$$
The function $v_0$ is smooth on the closed ball $\ov{B_{\rho}(0)}$ since near $\partial B_{\rho}(0)$ we have $v_0(x) = A^{1/\alpha}( \rho-|x|)$. Moreover the derivative of $v_0$ does not vanish on $\partial B_{\rho}(0)$.  Since $\tilde{w}$ is concave, the function $v_0$ is $\alpha$-concave.  Observe that in the quarter-sized ball $B_{\rho/4}(0)$ we have $\tilde{w}=w +A c_{\rho, \alpha}$ and hence the derivatives of $\tilde{w}$ at the origin coincide with those of $w$.

Let $v(t)$ be the solution to \eqref{PME2}  with initial condition $v(0)=v_0$.
The function $v(t)$ is smooth on the set $\{ v>0 \}$, including at $t=0$ (see the discussion in Section \ref{subsec} above).  Let $\lambda_1(x,t)$ be the largest eigenvalue of the matrix $D^2 v^{\alpha}(x,t)$.  We constructed $v_0$ so that $\lambda_1=0$ at $(x,t)=(0,0)$.   The matrix $D^2  v^{\alpha}$ at $(x,t)=(0,0)$ coincides with $D^2w$, for $w$ constructed in Lemma \ref{localinitialdata}.  In particular the eigenvalues of $D^2  v^{\alpha}$ at $(x,t)=(0,0)$ are distinct and so 
 the function $\lambda_1$ is smooth in a neighborhood of the origin for small $t>0$.  Also, all the spatial derivatives of $v^{\alpha}$ at $(x,t)=(0,0)$ coincide with  those of $w$ at the origin, and hence
$$\frac{\partial}{\partial t} \lambda_1(0,0) = \frac{\partial}{\partial t} (v^{\alpha})_{11}|_{(x,t)=(0,0)}>0,$$
from part (iii) of Lemma \ref{localinitialdata}.
Hence there exists a small $\delta>0$ such that $\lambda_1$ is strictly positive for $t \in (0,\delta)$, in a small neighborhood of the origin.  This establishes the theorem in this case.

For $\alpha=1$, we define $f$ instead using
$$f(r) = \left\{ \begin{array}{ll} 7\rho^2/8 & \quad 0\le r \le \frac{\rho}{4} \\ \rho^2-r^2 & \quad \frac{\rho}{2} \le r\le \rho, \end{array} \right.$$
and take $F(x) = f(|x|)$, $\tilde{w}=AF+\psi w$ and 
$v_0 = \tilde{w}$.  The rest of the argument follows similarly.
  
Finally, for  $\alpha=0$, we consider $f$ satisfying
$$f(r) = \left\{ \begin{array}{ll} \log ( \rho/2 ) + 1/4 & \quad 0\le r \le \frac{\rho}{4} \\ \log(\rho-r) & \quad \frac{\rho}{2} \le r\le \rho. \end{array} \right.$$
We define again $F(x) = f(|x|)$ but now take $A=1=\psi$ so that $\tilde{w}=F+ w$.  Note that $D^2 \tilde{w} <0$ on $B_{\rho}(0) \setminus \{ 0 \}$.  We define 
$$v_0 = e^{\tilde{w}},$$ which is log concave, smooth on $\ov{B_{\rho}(0)}$, is zero on $\partial B_{\rho}(0)$ and has nonvanishing derivative there.  
The rest of the argument is similar.
\end{proof}
 
 We remark that, as is made clear in the construction above, the initial data $v_0$ constructed in Theorem \ref{maintheorem0} is strictly $\alpha$-concave on the punctured ball $B \setminus \{0 \}$ in the sense that the matrix
 $$v_0 D_i D_j v_0 - (1-\alpha) D_i v_0 D_j v_0,$$
 is strictly negative definite on $B \setminus \{0 \}$.

\section{The Barenblatt and Graveleau solutions} \label{BG}

In this section we recall the construction of two well-known solutions to the porous medium equation:  Barenblatt solutions and Graveleau solutions.  The Barenblatt solutions (also known as Barenblatt-Prattle or ZKB solutions) are quadratic functions supported in expanding balls.  The Graveleau solutions are examples of ``focusing solutions'', and are described using a non-explicit  solution of an ODE.  For any fixed time $t<0$ the Graveleau solution is supported on the \emph{complement} of a ball in $\mathbb{R}^n$, which shrinks to a point as $t \rightarrow 0$. 

\subsection{Barenblatt solutions}
There is a one-parameter family of \emph{Barenblatt solutions}, whose pressure is given by
$$b=b^{(A)}(x,t) = t^{-\beta(m-1)} \left( A - \frac{\beta}{2n} t^{-2\beta/n} |x|^2 \right)_+$$
for a parameter $A>0$, where 
$$\beta = \frac{n}{n(m-1)+2}.$$
See for example \cite[Section 4.4]{V} for a derivation.
Note that $\beta(m-1)+ 2\beta/n=1$.  Observe that the Barenblatt solution is singular at $t=0$ but for any $\tau>0$ the time-translated solution $(x,t) \mapsto b(x, t+\tau)$ is a solution of the porous medium equation in the sense of Section \ref{subsec}.  

The next lemma, which is elementary and well-known, shows that by changing the constant $A$ we can find a Barenblatt solution $b^{(A)}(x,t)$ with prescribed radius of support and slope at the boundary, at some 
time $t_0$.

\begin{lemma}   \label{bl}
Given a slope $S>0$ and a radius $R>0$ we can find $A>0$ and $t_0 >0$ so that $x\mapsto b^{(A)}(x,t_0)$ has support of radius $R$ with $\lim_{|x| \rightarrow R^-} |\nabla b(x,t_0)|=S$. 
\end{lemma}
\begin{proof}
To ensure that the support of $b^{(A)}(x, t_0)$ has radius $R$, we need
\begin{equation} \label{A}
A = \frac{\beta}{2n} t_0^{-2\beta/n}R^2,
\end{equation}
On the other hand, as $|x| \rightarrow R^-$, we have
$$|\nabla b| \rightarrow \frac{\beta}{n}t_0^{-1}R.$$
Hence if we choose $t_0$ so that
$$S = \frac{\beta}{n} t_0^{-1} R,$$
namely
$$t_0 = \frac{\beta R}{Sn},$$
we can then define $A$ by
$$A = \frac{\beta}{2n} \left( \frac{\beta R}{Sn} \right)^{-2\beta/n} R^2,$$
so that $A$ satisfies
(\ref{A}).  This completes the proof.
\end{proof}

\subsection{Graveleau solutions} There is a one-parameter family of Graveleau solutions with pressure $g=g^{(c)}(x,t)$ for a parameter $c>0$ (see \cite{AA, AG}), defined as follows.  There are fixed constants $\alpha^* \in (1,2)$ and $\gamma<0$ depending only on $m$ and $n$ and a fixed solution $\varphi: [\gamma,0] \rightarrow [0, \infty)$ of a nonlinear degenerate ODE with the properties that 
\begin{enumerate}
\item[(i)] $\varphi(0)=0$ and $\varphi'(0)=-1$;
\item[(ii)] $\varphi(\gamma)=0$, $\varphi'(\gamma)>0$ is finite and  $\varphi>0$ on $(\gamma,0)$.
\end{enumerate}
We extend $\varphi$ to $(-\infty,0]$ by setting it equal to zero on $(-\infty, \gamma)$.  The Graveleau solutions are then given by 
$$g(x,t) = \frac{r^2 \varphi (c\eta)}{-t}, \quad \textrm{for }t<0.$$
where
$$r=|x|, \ \eta = tr^{-\alpha^*}$$
Note here that time $t$ and the variable $\eta$ are \emph{negative}, since the focusing time is $t=0$.  For any $\tau>0$ the translated solution $(x,t) \mapsto g(x,t-\tau)$ for $t \in [0,\tau)$ is locally a continuous solution of the PME in the sense of Section \ref{subsec}.   It is not a global solution, since it is not in $L^1(\mathbb{R}^n)$, but for our purposes we only need to know that the local comparison principle Theorem \ref{comparison} applies.

The free boundary of the solution is where $c\eta = \gamma$, namely
\begin{equation} \label{interface}
r = \left( \frac{ct}{\gamma} \right)^{1/\alpha^*}.
\end{equation}

We have the following elementary lemma, analogous to Lemma \ref{bl} above.

\begin{lemma} Given a slope $S>0$ and a radius $R>0$ we can find $c>0$ and $t_0<0$ so that $x \mapsto g^{(c)}(x, t_0)$ is supported on the complement of the ball of radius $R$, and $\lim_{|x| \rightarrow R^+} |\nabla g(r, t_0)| = S.$
 \end{lemma}
 \begin{proof}
 From (\ref{interface}) we have
 \begin{equation} \label{Rg}
 R = \left( \frac{ct}{\gamma} \right)^{1/\alpha^*}.
 \end{equation}
We compute $|\nabla g(x, t_0)|$ on the interface $c\eta=\gamma$ as
\[
\begin{split}
\lim_{|x| \rightarrow R^+} |\nabla g(x, t_0)| = {} & \frac{\partial}{\partial r}\bigg|_{r=R^+} \frac{r^2 \varphi(c\eta)}{-t}  \\
= {} & r^2 \frac{\varphi'(c\eta)}{-t_0} c \eta_r \quad \textrm{since $\varphi(c\eta)=0$}\\
= {} & \frac{r^2}{-t_0} \varphi'(\gamma) (-\alpha^* \eta r^{-1}) \\
={} & \frac{\alpha^* c\eta R}{t_0} \varphi'(\gamma) \\
= {} & \frac{\alpha^* \gamma R}{t_0} \varphi'(\gamma). 
\end{split}
\] 
Note that both $\gamma$ and $t_0$ are negative.  Then choose $t_0$ by
$$ \frac{\alpha^* \gamma R}{t_0} \varphi'(\gamma) =S,$$
namely
$$t_0 = \frac{\alpha^* \gamma R}{S} \varphi'(\gamma),$$
and $c$ by (\ref{Rg}) with $t=t_0$.  This proves the lemma.
 \end{proof}
 
 \section{Initial velocity of the boundary} \label{velocity}
 
For solutions of the PME which are smooth up to the free boundary for all times, the boundary moves in the outer normal direction with speed equal to $|\nabla v|$.  The next result shows that, under rather mild conditions on the initial data corresponding to interior and exterior sphere conditions, the boundary approximately moves with this speed for a short time.

\begin{theorem} \label{thmballs}
Let $v_0 \ge 0$ be a function on $\mathbb{R}^n$ such that the set $\Omega =\{ v_0> 0 \}$ is a bounded domain with boundary $\partial \Omega$.  Assume that $v_0|_{\ov{\Omega}}$ is in $C^1(\ov{\Omega})$ and $0 \in \partial \Omega$.
 Let $v(x,t)$ be the solution of the PME (\ref{PME2}) with initial data $v_0$, and denote $\Omega_t = \{ x\in \mathbb{R}^n \ | \ v(x,t)>0 \}$.

Write $B_r(p_r)$ for the open ball of radius $r$ centered at $p_r:=(0,\ldots, 0, r)$ and $B_R(q_R)$ for the open ball of radius $R$ centered at the point $q_R:=(0,\ldots,0, -R)$ (see Figure \ref{figure}).
\begin{enumerate}
\item[(i)] Suppose that $B_{r_0}(p_{r_0}) \subset \Omega$ for some $r_0>0$ and for some $S>0$,
\begin{equation} \label{v01}
| \nabla v_0(0)|  >S.
\end{equation}  
Then there exists $\delta>0$ and $r \in (0,r_0]$ such that 
$$B_{r+St}(p_r) \subset \Omega_t, \qquad \textrm{for } t \in (0, \delta).$$
\item[(ii)] Suppose that $B_{R_0}(q_{R_0}) \subset \Omega^c$ for some $R_0>0$ and for some $S>0$,
\begin{equation} \label{v02}
| \nabla v_0(0)| < S.
\end{equation}
Then there exists $\delta>0$ and $R \in (0, R_0]$  such that
$$B_{R-St}(q_R) \subset (\Omega_t)^c, \qquad \textrm{for } t \in (0, \delta).$$
\end{enumerate}
\end{theorem}
 \begin{proof}
 
 For (i), the assumptions imply that at the origin the gradient vector of $v_0$ is $(0, \ldots,0, (v_0)_n)$ with $(v_0)_n \ge S+\ve$ for some $\ve>0$.
We can find a constant $A>0$ and a time $t_0$ such that Barenblatt solution $b(x,t)=b^{(A)}(x,t)$ has the property that its support at time $t_0$ is the closure of the ball $B_r(p_r)$ with $0<r\le r_0$ and that $|\nabla b(x, t_0)| \rightarrow S+\ve/2$, as $x$ tends to a point on $\partial B_r(p_r)$.  Then 
shrinking $r$ if necessary, we may assume that $v_0(x) \ge b(x, t_0)$ on $B_{2r}(p_r)$.  Moreover, $b(x,t)=0$ on $\partial B_{2r}(p_r)$ for $t \in [t_0, t_0+\delta)$ for a small $\delta>0$.  It follows from the local comparison principle Theorem \ref{comparison} that $v(x,t) \ge b(x,t_0+t)$ on $B_{2r}(p_r)$ for $t\in[0, \delta)$.
But the support of $b(x,t_0+t)$ for $t \in (0,\delta)$, which is a ball centered at $p_r$, moves outward with speed greater than $S$, for at least a short time, and hence shrinking $\delta>0$ if necessary,
$$v(x,t)\ge b(x,t_0+t) >0 \ \textrm{for } x\in B_{r+St}(p_r), \ t \in (0,\delta),$$
 as required.
 
 In case (ii), the gradient vector of $v_0$ at the origin is $(0, \ldots, 0, (v_0)_n)$ with $(v_0)_n \le S -\ve$ for some $\ve>0$.  Then
 choose $c>0$ and a time $t_0>0$ so that the Graveleau solution $g(x,t) = g^{(c)}(x,t)$ has the following property:  $x \mapsto g(x,t_0)$ has support $(B_R(q_R))^c$ with $0<R\le R_0$ and  $| \nabla g(x, t_0)| \rightarrow S-\frac{\ve}{2}$ as $x$ in the support of $g$ tends to a point on $\partial B_R(q_R)$.
 It follows that, shrinking $R$ if necessary, we may assume that $v_0(x) \le g(x,t_0)$ on $B_{2R}(q_R)$.  Moreover, we may assume that $v_0(x) < g(x, t_0)$ for $x \in \partial B_{2R}(q_R)$.
 By continuity of $v(x,t)$ and $g(x,t)$, it follows that there exists $\delta>0$ such that $v(x,t) \le g(x, t_0+t)$ on $\partial B_{2R}(q_R)$ for $t\in [0, \delta)$.  Applying again the local comparison principle we obtain $v(x,t) \le g(x,t)$ on $B_{2R}(q_R)$ for $t \in [0, \delta)$.  The support of $g(x,t_0+t)$ is the complement of a ball centered at $q_R$ whose radius is shrinking at speed less than $S$ for $t\in [0,\delta)$ (again shrinking $\delta>0$ if necessary) we have,
 $$v(x,t) \le g(x, t_0+t)=0, \qquad \textrm{for } x \in B_{R-St}(q_R), \quad t \in (0,\delta),$$
 and this completes the proof.
  \end{proof}
 
  \begin{figure}
  \hspace{28pt}
 \begin{tikzpicture}
    \begin{axis}[thick,
        xmin=-1.2,xmax=1.2,
        ymin=-2,ymax=2,
       axis x line=middle,
       axis y line=middle,
        axis line style=->, xlabel style = {xshift=1.5cm, yshift=-.1cm},
        xlabel={$x_1, \ldots, x_{n-1}$}, 
        x tick label style={major tick length=0pt},
        ylabel={$x_n$},
        yticklabels={,,},xticklabels={,,},
        ]
         
        \addplot[no marks,-, ultra thick] expression[domain=-1.2:1.2,samples=100]{0.3*x^2+.05*x^4} 
                    node[pos=1,xshift=-0.8cm,yshift=0.5cm]{$\{ v_0>0 \}$} 
                    node[pos=1, xshift=-4.8cm, yshift=-.3cm]{$\partial \Omega$}; 
    \end{axis}
    \draw (3.425,3.56) circle (.7cm) node[xshift=.3cm]{$p_r$} node[xshift=-.7cm, yshift=0.9cm]{$B_r(p_r)$};
    \draw[ultra thick] (3.425, 3.56) circle (0.02cm);
    \draw (3.425,1.725) circle (1.1cm) node[xshift=.34cm]{$q_R$} node[xshift=-1.2cm, yshift=-1.15cm]{$B_R(q_R)$};
    \draw[ultra thick] (3.425,1.725) circle (0.02cm);
\end{tikzpicture}
\caption{\ }
\label{figure}
 \end{figure}

 An immediate consequence is that if we have both interior and exterior ball conditions then the boundary initially moves with velocity equal to the gradient of $v_0$ in the sense of the following:
 
 \begin{corollary}
 Let $v_0 \ge 0$ be a  function on $\mathbb{R}^n$ such that the set $\Omega =\{ v_0> 0 \}$ is a bounded domain with boundary $\partial \Omega$.  Assume that $v_0|_{\ov{\Omega}}$ is in $C^1(\ov{\Omega})$ and $0 \in \partial \Omega$.
Let $v(x,t)$ be the solution of the PME (\ref{PME2}) with initial data $v_0$, and denote $\Omega_t = \{ x\in \mathbb{R}^n \ | \ v(x,t)>0 \}$.  Suppose that there exists $r_0, R_0>0$ such that $B_{r_0}(p_{r_0}) \subset \Omega$ and $B_{R_0}(q_{R_0}) \subset \Omega^c$, using the notation of the previous theorem.

Let $(0, \ldots, 0, y(t))$ be a point on $\partial \Omega$, where $y(t)$ for $t\ge 0$ is defined by
$$y(t) = \max \{ x_n \le 0 \ | \ v(0, \ldots, 0,x_n,t)=0 \}.$$
Then the one-sided derivative $y'(0)$ exists and is equal to $-|\nabla v_0|(0)$.
 \end{corollary}
 \begin{proof}
Let $\ve>0$.  Then (\ref{v01}) holds for $S= |\nabla v_0|-\ve$ and (\ref{v02}) holds for $S=|\nabla v_0| +\ve$, where we are evaluating $|\nabla v_0|$ at the origin.
Applying Theorem \ref{thmballs},   there exists $\delta, r, R>0$ such that 
$$B_{r+(|\nabla v_0|-\ve)t}(p_r) \subset \Omega_t, \quad B_{R-(|\nabla v_0|+\ve)t}(q_R) \subset (\Omega_t)^c, \quad \textrm{for } t\in (0,\delta).$$
Then for $t\in (0,\delta),$
$$-(|\nabla v_0|+\ve)t \le y(t) \le - (|\nabla v_0|-\ve)t,$$
namely, since $y(0)=0$,
$$\left| \frac{y(t)-y(0)}{t} + | \nabla v_0| \right| < \ve.$$
Hence the one-sided derivative $y'(0)$ exists and is equal to $-|\nabla v_0|$.
 \end{proof}

 \section{Concavity breaking on the boundary}\label{proofmainthm}
 
 In this section we prove Theorem \ref{maintheorem}.
Write $v_0$ for the function constructed by Theorem \ref{thm1}.  Since $\displaystyle{x\mapsto \frac{\partial v_0}{\partial y}(x,0)}$ is positive and strongly convex as a function of $x \in [-1,1]$ we have
$$\frac{\partial v_0}{\partial y}(0,0) < \frac{ \frac{\partial v_0}{\partial y} (-1/2, 0) + \frac{\partial v_0}{\partial y}(1/2,0)}{2}.$$
Hence we can find positive constants $S_-, S_0, S_+$ such that
\begin{equation}\label{Ss}
S_0< \frac{S_-+S_+}{2},
\end{equation}
and 
$$
\frac{\partial v_0}{\partial y} (-1/2,0) >  S_-,\quad 
\frac{\partial v_0}{\partial y} (0,0) <  S_0 , \quad
\frac{\partial v_0}{\partial y} (1/2, 0)  >  S_+.
$$
Let $v(x,t)$ be the solution of the PME (\ref{PME2}) starting at $v_0$.
We apply Theorem \ref{thmballs} at each of the three points $(-1/2, 0)$, $(0,0)$, $(1/2,0)$ to find a small $\delta>0$ such that for $t \in (0,\delta)$,
$$(-1/2, -S_- t) \in \ov{\Omega_t}, \ (0, -S_0t) \notin \Omega_t, \ (1/2, - S_+t) \in \ov{\Omega_t},$$
But if $\Omega_t$ were convex for $t \in [0, t_0]$, for some $t_0>0$, then we must have
$$S_0 \ge \frac{S_- + S_+ }{2},$$
contradicting (\ref{Ss}).  This completes the proof of Theorem \ref{maintheorem}.

\end{document}